\documentclass[oneside,reqno,10pt]{amsart}
\usepackage{graphicx,amsthm,amssymb,amsmath,tikz,amsfonts} 
\usepackage{booktabs,caption}
\usepackage{bm}
\usepackage{enumitem}
\usepackage{float}
\usepackage{siunitx}

\usepackage[symbol]{footmisc}

\usepackage{hyperref}

\numberwithin{equation}{section} 

\newcommand{\Mod}[1]{\ (\mathrm{mod}\ #1)}
\newcommand{\RNum}[1]{\uppercase\expandafter{\romannumeral #1\relax}}

\theoremstyle{plain} %--default
\newtheorem{theorem}{Theorem}[section]

\newtheorem{lemma}{Lemma}[section]
\newtheorem*{thm}{Theorem}

\theoremstyle{definition}

\newtheorem{rek}{Remark}
\newtheorem{defn}{Definition}[section]

\newcommand{\bbn}{\mathbb{N}}
\newcommand{\bbz}{\mathbb{Z}}
\newcommand{\ds}{\displaystyle}

\newcommand{\li}{\operatorname{li}}

\newcommand{\imply}{\Longrightarrow\qquad}
\newcommand{\bburl}[1]{\textcolor{blue}{\url{#1}}}

\numberwithin{equation}{section}

\begin{document}
\author{Muhammet Boran}
\address{Department of Mathematics, Yıldız Technical University, 34220 Esenler, Istanbul, TURKEY}
\email{\bburl{muhammet.boran@std.yildiz.edu.tr}}

\author{John Byun}
\address{ Department of Mathematics and Statistics, Carleton College, One North College Street, Northfield, MN 55057, USA}
\email{\bburl{byunj@carleton.edu}}

\author{Zhangze Li}
\address{ Department of Mathematics, University of Michigan, Ann Arbor, MI 48109, USA}
\email{\bburl{zhangzli@umich.edu}}

\author{Steven J. Miller}
\address{Department of Mathematics and Statistics, Williams College, Williamstown, MA 01267, USA}
\email{\bburl{sjm1@williams.edu}}

\author{Stephanie Reyes}
\address{Institute of Mathematical Sciences, Claremont Graduate University, Claremont, CA 91711, USA}
\email{\bburl{stephanie.reyes@cgu.edu}}

\title{Sums of Powers of Primes in Arithmetic Progression}

\thanks{This work was done in the 2023 Polymath Jr. program and was partially supported by NSF award DMS-2113535.}

\maketitle
\noindent

\begin{abstract}
Gerard and Washington proved that, for $k > -1$, the number of primes less than $x^{k+1}$ can be well approximated by summing the $k$-th powers of all primes up to $x$. We extend this result to primes in arithmetic progressions: we prove that the number of primes $p\equiv n \pmod m$ less than $x^{k+1}$ is asymptotic to the sum of $k$-th powers of all primes $p\equiv n \pmod m$ up to $x$. We prove that the prime power sum approximation tends to be an underestimate for positive $k$ and an overestimate for negative $k$, and quantify for different values of $k$ how well the approximation works for $x$ between $10^4$ and $10^8.$

\end{abstract}

\section{Introduction}
\subsection{Historical Context} 

Approximating the distribution of prime numbers is a major motivating question in number theory. This task can be broken down into counting primes up to a given number and estimating where prime numbers appear within the set of natural numbers. The latter is a much more difficult task and falls outside of the scope and context of our results. We instead extend a recent approximation for the prime counting function to primes in arithmetic progressions. We use analogous methods to the work of Gerard and Washington \cite{GW}, incorporating the findings from Page \cite{P} and Siegel \cite{Si} and applying the prime number theorem for arithmetic progressions where Gerard and Washington used the prime number theorem. While our methods are similar, extending them to primes in arithmetic progressions allows to investigate finer questions on the distributions of prime numbers, for example, by studying Chebyshev bias.

\begin{defn}
For $k\in \mathbb{R}$, define
\begin{equation*}
    \pi(x) \ := \ \sum\limits_{p\le x} 1\qquad \text{and}\qquad \pi_k(x) \ := \ \sum\limits_{p\le x} p^k,
\end{equation*}
where $p$ denotes a prime number here and throughout; $\pi(x)$ is the \emph{prime-counting function}. 
\end{defn}
\begin{defn}
    For real $x\ge 0$, define the \emph{logarithmic integral} 
    \begin{equation*}
        \li(x)\ :=\ \int_0^x \frac{dt}{\log t}.
    \end{equation*}
\end{defn}

\begin{defn}
    Let $f$ and $g$ be two functions defined on $\mathbb{R}_{>0}$, and $g(x)$ be strictly positive for all large enough values of $x$. We say
    \[f(x)\ = \ O(g(x)) \qquad \text{as} \qquad x\to\infty\]
    if there exist constants $M$ and $x_0$ such that
    \[|f(x)|\ \le \ M|g(x)| \qquad\text{for all}\qquad x\ \ge \ x_0. \]
    We also say
    \[f(x)\ll g(x)\qquad\text{is equivalent to}\qquad f(x)=O(g(x)).\]
\end{defn}
\begin{defn}\label{small-o}
    Let $f$ and $g$ be two functions defined on $\mathbb{R}_{>0}$, and $g(x)$ be strictly positive for all large enough values of $x$. We say
    \begin{equation*}
        f(x)\ = \ o(g(x)) \qquad \text{as} \qquad x\to\infty
    \end{equation*}
    if for every positive constant $\varepsilon$, there exists a constant $x_0$ such that
    \[|f(x)|\ \le \ \varepsilon g(x)\qquad\text{for all}\qquad x\ \ge \ x_0 \]
    and, equivalently,
    \[\lim_{x\to\infty} \frac{f(x)}{g(x)}\ = \ 0.\]
    We also say
    \[f(x)\lll g(x)\qquad\text{is equivalent to}\qquad f(x)=o(g(x)).\]
\end{defn}
\begin{defn}
    Given two functions $f$ and $g$, we say
    \[f(x)\sim g(x) \qquad \text{as} \qquad x\to\infty,\]
    read $f(x)$ \emph{is asymptotic to} $g(x)$, if
    \[\lim_{x\to\infty}\frac{f(x)}{g(x)}\ =\ 1\]
    and, equivalently, that
    \[f(x)\ =\ g(x)+O(h(x)).\]
\end{defn}

\noindent
In 1896, Vallée-Poussin \cite{V-P} proved that 
\begin{equation*}
    \pi(x) \sim \li(x) \qquad\text{as}\qquad x\to \infty,
\end{equation*}
which is known as the \textit{prime number theorem}, using techniques which relied on complex analysis and the Riemann $\zeta$-function. Vallée-Poussin \cite{V-P} also estimated the error term in the prime number theorem by proving that
\begin{equation}\label{prime-counting}
    \pi(x)\ =\ \li(x)+O\left(x \exp{\left(-c\sqrt{\log x}\right)}\right)\ \ (x\geq 2)
\end{equation}
where $c$ is a positive constant. Further improvements on the error term began to emerge. In 1922, Littlewood (outlined in \cite{Ku}) proved that
\begin{equation*}
    \pi(x)-\li(x) \ll x \exp{\left( -c \sqrt{\log x\log \log x}\right)}\ \ (x\geq 2).
\end{equation*}
In 1958, Korobov \cite{Ko} and Vinogradov \cite{V} proved that
\begin{equation*}
    \pi(x)-\li(x) \ll x \exp{ \left( -c (\log x)^{3/5} (\log \log x)^{-1/5} \right) }\ \ (x\geq 2).
\end{equation*}

\noindent Since then, various alternative proofs have been given which do not use complex analysis, with perhaps the most renowned being the proof by Erd\H{o}s and Selberg \cite{Se}.\\

\noindent A natural immediate extension of approximating the number of primes up to a certain real number is estimating the \emph{sum of primes} up to a certain real number. Using Vall\'ee-Poussin's result [Equation (\ref{prime-counting})], Szalay \cite[Lemma 1]{Sz} found the following approximation for the sum of primes less than a real number $x\geq 0$:
\begin{equation}
    \pi_1(x) \ =\ \li\left(x^2\right)+O\left(x^2 \exp{\left(-c\sqrt{\log x}\right)} \right).\label{first-est-k=1}
\end{equation}

\noindent
In 1996, Massias and Robin \cite[p. 217]{MR} approximated the sum of prime numbers less than a real number $x\geq 0$ as 
\begin{equation*}
    \pi_1(x)\ =\ \li\left(\li^{-1}(x)^2\right)+O\left(x^2 \exp{\left(-\gamma\sqrt{\log x}\right)} \right)
\end{equation*} 
where $\gamma$ is a positive constant and $\li^{-1}(x)$ denotes the inverse of $\li(x)$.\\

\noindent Massias and Robin \cite[Théorème D(v)]{MR} further bounded $\pi_1(x)$ above for every $x \geq 24281$ by
\begin{equation}\label{ineq}
    \pi_1(x)\  \leq\  \frac{x^2}{2 \log x} + \frac{3x^2}{10 \log^2 x}.
\end{equation}

\noindent
Around three decades following the publication of Szalay's result [Equation (\ref{first-est-k=1})], Axler \cite{Ax} applied it to show that 
\begin{equation*}
    \pi_1(x) \ =\ \frac{x^2}{2\log x}+\frac{x^2}{4\log^2 x}+\frac{x^2}{4\log^3 x}+\frac{3x^2}{8\log^4 x}+O\left(\frac{x^2}{\log^5 x}\right).
\end{equation*}
Axler \cite[Theorem 1.1]{Ax} also improved on the upper bound given by Massias and Robin [Line (\ref{ineq})] to give the following inequality 
\begin{equation*}
    \pi_1(x)\  <\ \frac{x^2}{2\log x}+\frac{x^2}{4\log^2 x}+\frac{x^2}{4\log^3 x}+\frac{5.3x^2}{8\log^4 x},
\end{equation*}
which holds for every $x \geq 110118925$.
For $x\geq 905238547$, Axler \cite[Theorem 1.2]{Ax} also gives a lower bound of $\pi_1(x)$, which differs from the upper bound only by a constant in the fourth term from the upper bound.
\begin{equation*}
    \pi_1(x) \  >\  \frac{x^2}{2\log x}+\frac{x^2}{4\log^2 x}+\frac{x^2}{4\log^3 x}+\frac{1.2x^2}{8\log^4 x}.
\end{equation*}

\noindent
In 2016, shortly following Axler's preprint, Gerard and Washington \cite[Theorem 1]{GW} generalized these results to sums of fixed powers $k$ of primes. They proved that $\pi_k(x)$ is asymptotic to $\pi(x^{k+1})$:
    \begin{align*}
        \pi_k&(x)-\pi(x^{k+1}) \\
        &=
        \begin{cases}
            O\left( x^{k+1} \exp \left( -A(\log x)^{3/5} (\log \log x)^{-1/5}\right) \right) & \text { if } k>0,\\ 
            O\left(x^{k+1} \exp \left( -(k+1)^{3/5}A(\log x)^{3/5} (\log \log x)^{-1/5}\right) \right) & \text { if } -1<k<0,
        \end{cases}
    \end{align*}
    where $A=.2098$. Their result is significant both from the perspective of studying the prime counting function and from the perspective of studying sums of powers of primes, as it relates the two functions with small relative error for large $x$. From the perspective of studying the prime counting function, their result allows for approximating $\pi(x)$ by $\pi_k\left(x^{1/(k+1)}\right)$, which can significantly reduce the number of primes used in the approximation.
    
\subsection{Statement of Main Results}
Throughout this paper, we take $m,n\in\bbz_{>0}$ to be two coprime positive integers, specifically with $n$ a unit in $\bbz/m\bbz$. As mentioned previously, we use $p$ to denote prime numbers throughout.\\

\noindent To extend previous results, we first define the sum of $k$-powers of primes for arithmetic progressions.
    \begin{defn}
        For a fixed $k\in \mathbb{R}$, define
       \[\pi(x;m,n) \ = \ \sum\limits_{\substack{p\le x\\ p\equiv n\pmod{m}}} 1\qquad \text{and}\qquad \pi_k(x;m,n) \ = \ \sum\limits_{\substack{p\le x\\ p\equiv n\pmod{m}}} p^k.\]
    \end{defn}

    \noindent
    Also, recall the definition of Euler’s totient function: 
    \[\varphi(m)\ =\ \#\{x\in\bbn\ :\ x<m\ \text{and}\ \gcd(x,m)=1\},\]
    which counts the number of positive integers up to $m$ which are coprime to $m$.\\
    
    \noindent We extend the work of Gerard and Washington \cite{GW}, studying sums of powers of primes \emph{in arithmetic progressions}. Dirichlet's theorem on arithmetic progressions states that there are infinitely many primes $p$ in the form $p\equiv n\Mod m$. Thus, $\pi(x;m,n)$ grows to $\infty$ as $x$ approaches $\infty$. Also, Vallée-Poussin \cite{V-P} proved the \emph{prime number theorem on arithmetic progressions}, concluding that
    \begin{equation}
        \pi(x;m,n)\ \sim\ \frac{\pi(x)}{\varphi (m)}\ \sim\ \frac{x}{\varphi (m) \log x}  \qquad\text{as}\qquad x\to \infty,\label{dric:eqn}
    \end{equation}
    where $\gcd(m,n)=1$. A proof of (\ref{dric:eqn}) is outlined in \cite[p. 154-155]{Ap}. In 1935, Page \cite{P} proved that there exists a positive constant $\delta$ such that 
    \begin{equation*}
        \pi(x;m,n)\ =\ \frac{\li(x)}{\varphi (m)}+ O\left(x \exp{\left(-(\log x)^{\delta}\right)}\right)\ \ (x\geq 2), 
    \end{equation*}
    whenever $1\leq m\leq (\log x)^{2-\delta}$. In the same year, Siegel \cite{Si} proved that there exists a positive constant $\xi_\eta$ depending on $\eta$ such that
    \begin{equation*}
        \pi(x;m,n)\ =\ \frac{\li(x)}{\varphi (m)}+ O\left(x \exp{\left(-\frac{\xi_{\eta}}{2}\sqrt{\log x}\right)}\right)\ \ (x\geq 2),
    \end{equation*}
    whenever $1\leq m\leq (\log x)^{\eta}$. We use the above version of the prime number theorem on arithmetic progressions to obtain our main result.\\
    
    \noindent Following the methods of Gerard and Washington in \cite{GW} we obtain the following theorem.
    \begin{theorem}
    \label{1:thm}
    Fix a real number $k>-1$ and positive integers $m,n\in\bbz_{>0}$ such that $\gcd(m,n)=1$. Then we can approximate the number of primes $p\equiv n\pmod{m}$ less than $x^{k+1}$ by the sum of $k$-powers of primes $p\equiv n\pmod{m}$ less than a real number $x$:
    \begin{align*}
        \pi_k&(x;m,n)-\pi(x^{k+1};m,n) \\
        &=
        \begin{cases}
            O\left(x^{k+1} \exp\left(-\frac{1}{2}\alpha \sqrt{\log x}\right)\right) & \text { if } k>0,\\ 
            O\left(x^{k+1} \exp\left(-\frac{1}{2}\alpha \sqrt{(k+1)\log x}\right)\right) & \text { if } -1<k<0,
        \end{cases}
    \end{align*}
    where $\alpha$ is a positive constant.
    \end{theorem}

\noindent    
Additionally, generalizing the proof of \cite[Theorem 2]{LW}, we prove the following.
\begin{theorem}\label{thm:error-pi-mn}
    Fix a real number $k>-1$ and positive integers $m,n\in\bbz_{>0}$ such that $\gcd(m,n)=1$. Then
    \begin{equation*}
        \int_1^\infty \frac{\pi_k(t;m,n)-\pi(t^{k+1};m,n)}{t^{k+2}}\, dt\ =\ -\frac{\log(k+1)}{(k+1) \varphi(m)}.
    \end{equation*}
    
    \noindent Further,
    \begin{align*}
    -\frac{\log(k+1)}{(k+1) \varphi(m)}\ &<\ 0\qquad (k>0)\\
    -\frac{\log(k+1)}{(k+1) \varphi(m)}\ &>\ 0\qquad (-1<k<0).\\
    \end{align*}
    \end{theorem}

\begin{rek}\label{remark_1}
    In 1853, Chebyshev \cite{C} noticed that there are more primes of the form $p\equiv 3 \pmod 4$ than $p\equiv 1 \pmod 4$. Over the years the synonymous expression “prime number races” has emerged to describe problems originated from Chebyshev. Rubinstein and Sarnak \cite{RS} provided a framework for the quantification of Chebyshev’s bias in prime number races in arithmetic progressions. They considered the sets $P(m;n_1,n_2) = \{x \geq 2: \pi(x; m, n_1) > \pi(x; m, n_2)\}$. With an appropriate notion of size, under standard assumptions, they proved that $P(4;3,1)\ >\ P(4;1,3)$.\\
    
    \noindent In Appendix \ref{error-cal}, we provide calculations of various examples for different values of $k$, $m$, and $n$. The integral computation in Theorem \ref{thm:error-pi-mn} confirms the observations from our data that 
    \[\pi_k(x;m,n)-\pi(x^{k+1};m,n)\] 
    tends to be negative when $k > 0$, and tends to be positive when $-1 < k < 0$ for large values of $x$. Note that this does not definitively specify the sign of the error for given large values of $x$. Since our result integrates over all reals greater than one, we can only conclude that the ``net'' sign will either be negative or positive for specific $k$. A more refined approach to evaluating $\pi_k(x;m,n)-\pi(x^{k+1};m,n)$ when $x$ is sufficiently large necessitates the utilization of the Riemann $\zeta$-function and the Riemann Hypothesis to analyze the bias between $\pi(x;m,n)$ and $\li(x)$, which allows for a more precise determination of the error's behavior in such cases.
    
\end{rek}

\section{Proofs of Theorems}
\noindent One useful technique used to prove Theorem \ref{1:thm} and Theorem \ref{thm:error-pi-mn} is Riemann-Stieltjes Integration. The detailed procedure is included in Appendix \ref{append-int}.\\

\subsection{Proof of Theorem \ref{1:thm}}
We divide the proof in three parts. 
\subsubsection{Start of Proof}
\begin{proof}
By the prime number theorem for arithmetic progressions \cite[Theorem 7-25]{WJL},
\begin{equation}\label{eqn:3}
    \pi(x;m,n)\quad =\quad \frac{1}{\varphi (m)} \int_{2}^{x} \frac{\, dt}{\log t}\ +\ O \left(x \exp\left(-\frac{1}{2}\alpha \sqrt{\log x}\right)\right) \ (x\ge 2)
\end{equation}
where $\alpha$ is a positive constant. In the rest of the proof, denote 
\begin{equation}
    \label{E-error}
    E(x)\ =\ O \left(x \exp\left(-\frac{1}{2}\alpha \sqrt{\log x}\right)\right)
\end{equation}
as the error term in Equation (\ref{eqn:3}).\\

\noindent By applying Equation (\ref{riea-stieljes}) on Riemann-Stieltjes integration, for $x\geq2$ we have 
\begin{equation}\label{arh:prog}
    \pi_k(x;m,n)\quad =\sum_{\substack{p \le x\\ p\  \equiv\  n \Mod{m}}}p^k\quad =\quad \int_{2}^{x} t^k\,  d\pi(t;m,n).
\end{equation}
Taking differentials yields
\begin{align*}
    d\pi(x;m,n)\ &=\ d\left(\frac{1}{\varphi (m)} \int_{2}^{x} \frac{\, dt}{\log t}\right) + dE(x)\\
    &=\ \frac{1}{\varphi (m)} \frac{\, dt}{\log t} + dE(x).
\end{align*}
 Substituting $d\pi(x;m,n)$ into (\ref{arh:prog}), we obtain
\begin{align*}
    \pi_k(x;m,n)\ &=\  \int_{2}^{x} t^k \left(\frac{1}{\varphi (m)} \frac{\, dt}{\log t} + dE(t)\right)\\
    &=\ \frac{1}{\varphi(m)} \int_{2}^{x} \frac{t^k}{\log t} \, dt+ \int_{2}^{x} t^k\,  dE(t).
\end{align*}
The change of variable $t\mapsto t^{1/(k+1)}$ yields
\begin{align*}
    \frac{1}{\varphi(m)}\int_{2}^{x} \frac{t^k}{\log t} \, dt\ &=\ \frac{1}{\varphi(m)}\int_{2^{k+1}}^{x^{k+1}} \frac{t^{k/(k+1)}}{\log \left(t^{1/(k+1)}\right)}\frac{t^{-k/(k+1)}}{k+1}\, dt\\
    &=\ \frac{1}{\varphi(m)}\int_{2^{k+1}}^{x^{k+1}} \frac{\, dt}{\log t}\\ 
    &=\ \pi(x^{k+1};m,n)- E(x^{k+1})-C
\end{align*}
where $C=\li(2^{k+1})/\varphi(m)$.\\

\noindent Thus, 
\begin{equation*}
    \frac{1}{\varphi(m)} \int_{2}^{x} \frac{t^k}{\log t} \, dt\ =\ \pi(x^{k+1};m,n)- E(x^{k+1})-C.
\end{equation*}\\

\noindent Also, integration by parts yields
\begin{align*}
    \int_{2}^{x} t^k \, dE(t)\ &=\ t^k E(t)\Big|_{2}^{x} -k\int_{2}^{x} t^{k-1} E(t) \, dt\\
    &=\ x^k E(x)-k\int_{2}^{x} t^{k-1} E(t) \, dt.
\end{align*}
Finally, we obtain
\begin{align}
    \pi_k(x;m,n)\ =\ &\pi(x^{k+1};m,n)-E(x^{k+1})-C+x^k E(x)-k\int_{2}^{x} t^{k-1} E(t) \, dt.\label{anathrm}
\end{align}

\subsubsection{Error Term}
We now simplify (\ref{anathrm}).\\

\noindent First, according to (\ref{E-error}), 
\begin{equation*}
    E(x)\ =\ O \left(x \exp\left(-\frac{1}{2}\alpha \sqrt{\log x}\right)\right).
\end{equation*}
\noindent Thus,
\begin{equation*}
    E(x^{k+1})\ =\ O\left(x^{k+1} \exp\left(-\frac{1}{2}\alpha \sqrt{(k+1)\log x}\right)\right).
\end{equation*}\\
Also, we have
\begin{align*}
    x^kE(x)\ &=\ x^k\ O\left(x \exp\left(-\frac{1}{2}\alpha \sqrt{\log x}\right)\right)\\ 
    &=\ O\left(x^{k+1} \exp\left(-\frac{1}{2}\alpha \sqrt{\log x}\right)\right).
\end{align*}
The next step is to bound
\begin{equation}
\label{int-error}
    k \int_{2}^{x} t^{k-1} E(t)\, dt
\end{equation}
by using \cite[Lemma 1]{GW}. By applying (\ref{E-error}), we can bound (\ref{int-error}) as
\begin{align}
    \nonumber
    k \int_{2}^{x} t^{k-1} E(t)\, dt\ &\le \ D  \int_{2}^{x} t^{k-1}\left( t \exp\left(-\frac{1}{2}\alpha \sqrt{\log t}\right)\right)\, dt\\
    \label{int-error-2}
    &=\ D  \int_{2}^{x} t^k \exp\left(-\frac{1}{2}\alpha \sqrt{\log t}\right)\, dt
\end{align}
where $D$ is a positive constant.\\

\noindent The integrand of (\ref{int-error-2}) is of the form\quad $\ds t^{k} \exp(-f(t))$\quad where\quad $\ds f(t)=\frac{1}{2}\alpha \sqrt{\log t}$.\\ 
To apply \cite[Lemma 1]{GW}, we verify that
\begin{enumerate}[label=\roman*.]
    \item $\ds \lim_{t\to\infty} [t^{k+1} \exp(-f(t))]\ =\ \infty$
    \item and\qquad $\ds \lim_{t\to\infty} [t f'(t)]\ =\ 0$
\end{enumerate}
for $k>-1$ in separate computations included in Appendix \ref{append-lim}. 
Thus, applying \cite[Lemma 1]{GW}, we obtain 
\begin{equation*}
    k \int_{2}^{x} t^{k-1} E(t)\, dt\ = \ O\left(x^{k+1} \exp\left(-\frac{1}{2}\alpha \sqrt{\log x}\right)\right).
\end{equation*}
Finally, we reduce
\begin{equation*}
    O\left(x^{k+1}\exp\left(-\frac{1}{2}\alpha \sqrt{\log x}\right)\right)-C \ = \ O\left(x^{k+1} \exp\left(-\frac{1}{2}\alpha \sqrt{\log x}\right)\right).
\end{equation*}
\subsubsection{Conclusion}
    Based on previous sections, we simplify (\ref{anathrm}) as
    \begin{align*}
        \pi_k&(x;m,n)-\pi(x^{k+1};m,n) \\
        &=
        \begin{cases}
            O\left(x^{k+1} \exp\left(-\frac{1}{2}\alpha \sqrt{\log x}\right)\right) & \text { if } k>0,\\ 
            O\left(x^{k+1} \exp\left(-\frac{1}{2}\alpha \sqrt{(k+1)\log x}\right)\right) & \text { if } -1<k<0.
        \end{cases}
    \end{align*}
Therefore, we conclude that
\begin{equation*}
     \pi_k(x;m,n)\ \sim\ \pi(x^{k+1};m,n).
\end{equation*}
\end{proof}

\subsection{Proof of Theorem \ref{thm:error-pi-mn}}
Before proving Theorem \ref{thm:error-pi-mn}, we require additional results.
\subsubsection{Abel's Summation Formula and an Additional Lemma}
\emph{Abel's summation formula} is an important tool in analytic number theory derived from integrating Riemann-Stieltjes by parts. The statement of the formula and its hypotheses are given below.\\

\noindent First, recall that a function whose domain is the positive integers is called an \textit{arithmetic function}.

\begin{theorem}[Abel's Summation Formula]\label{thm:abel-id}
     Let $a$ be an arithmetic function. Define
     \begin{equation*}
         A(t)\ :=\ \sum\limits_{\substack{1\le n\le t}} a(n)
     \end{equation*}
     for $t\in \mathbb{R}$. Fix $x,y\in \mathbb{R}$ such that $x<y$ and let $f$ be a continuously differentiable function on $[x,y]$. Then
     \begin{equation}
         \sum\limits_{\substack{x< n\le y}} a(n) f(n)\ =\ A(y)f(y)-A(x)f(x)-\int_{x}^{y} A(t) f'(t)\, dt.
     \end{equation}
\end{theorem}
\noindent The following lemma can be found in $\textit{Davenport's}$ Multiplicative Number Theory textbook \cite{D}. We provide a proof for that.
\begin{lemma}\label{lem:error-pi-mn}
    Fix positive integers $m,n\in\bbz_{>0}$ such that $\gcd(m,n)=1$. Then for $x\ge 2$,
    \[ \sum\limits_{\substack{p\le x\\ p\equiv n\pmod{m}}} \frac{1}{p}\ =\ \frac{\log\log (x)}{\varphi(m)}+B+O\left(\frac{1}{\log x}\right),\]
    where $B$ is a constant.
\end{lemma}
\begin{proof}
Let 
\begin{equation*}
    f(p)\ =\ \frac{1}{\log p}.
\end{equation*}
 Note that $f$ is defined on primes and is continuous on  $[2,x]$. Let
\begin{equation*}
a(p)\ =\
        \begin{cases}
            \frac{\log p}{p} & \text{ if}\quad p\ \equiv\  n\pmod{m}\quad \text{is prime},\\ 
            0 & \text{otherwise.}
        \end{cases}
\end{equation*}
If $x\ge 2$, 
\begin{equation*}
    f'(x)\ =\ -\frac{1}{x\log^2 x}.
\end{equation*}
Define 
\begin{equation*}
    A(x)\ :=\ \sum\limits_{\substack{p\le x\\ p\equiv n\pmod{m}}} a(p)\ =\ \sum\limits_{\substack{p\le x\\ p\equiv n\pmod{m}}} \frac{\log p}{p}.
\end{equation*}
According to \cite[p.148, Theorem 7.3]{Ap}
\begin{equation*}
    \sum\limits_{\substack{p\le x\\ p\equiv n\pmod{m}}} \frac{\log p}{p}\ =\ \frac{\log x}{\varphi(m)}+ R(x),\ \text{where}\ R(x)\ =\ O(1).
\end{equation*}
Applying Abel's Summation Formula (Theorem \ref{thm:abel-id}), we have
\begin{align}
\nonumber
    \sum\limits_{\substack{p\le x\\ 
    p\equiv n\pmod{m}}}\frac{1}{p}\ 
    &=\ \sum\limits_{\substack{1<p\le x\\ p\equiv n\pmod{m}}} a(p) f(p)\\
    \nonumber
    &=\ \frac{1}{\log x}\left(\frac{\log x}{\varphi(m)}+ O(1) \right)+\int_{2}^{x} \frac{1}{t\log^2 t}\left(\frac{\log t}{\varphi(m)}+ R(t)\right) \,dt\\
    &=\ \frac{1}{\varphi(m)}
    +O\left(\frac{1}{\log x}\right)+\frac{1}{\varphi(m)}\int_{2}^{x} \frac{dt}{t \log t} + \int_{2}^{x} \frac{R(t)}{t \log^2 t} \,dt.\label{contin}
\end{align}
Rewrite
\begin{equation*}
    \int_{2}^{x} \frac{R(t)}{t \log^2 t} dt\ =\ \int_{2}^{\infty} \frac{R(t)}{t \log^2 t} \,dt -\int_{x}^{\infty} \frac{R(t)}{t \log^2 t} \,dt.
\end{equation*}
Notice that $\int_{2}^{\infty} \frac{R(t)}{t \log^2 t} \,dt$ exists because $R(t)=O(1).$\\

\noindent Further, since $R(t)=O(1)$, we have
\begin{equation*}
    \int_{x}^{\infty} \frac{R(t)}{t \log^2 t} \, dt\ =\ O\left(\int_{x}^{\infty} \frac{1}{t \log^2 t} \,dt\right) \ =\ O\left(\frac{1}{\log x}\right).
\end{equation*}
Continuing from (\ref{contin}):
\begin{align*}
    \sum\limits_{\substack{p\le x\\ 
    p\equiv n\pmod{m}}}\frac{1}{p}\
    &=\ \frac{1}{\varphi(m)}+O\left(\frac{1}{\log x}\right)+\frac{\log \log x}{\varphi(m)}-\frac{\log \log 2}{\varphi(m)}\\
    &+\int_{2}^{\infty} \frac{R(t)}{t \log^2 t} \,dt -\int_{x}^{\infty} \frac{R(t)}{t \log^2 t} \,dt\\
    &=\ \frac{\log \log x}{\varphi(m)}+B+ O\left(\frac{1}{\log x}\right),
\end{align*}
where
\begin{equation*}
    B\ =\ \frac{1}{\varphi(m)}-\frac{\log \log 2}{\varphi(m)}+\int_{2}^{\infty} \frac{R(t)}{t \log^2 t} \,dt.
\end{equation*}
Therefore, we conclude that
\begin{equation*}
 \sum\limits_{\substack{p\le x\\ p\equiv n\pmod{m}}}\frac{1}{p}\    =\ \frac{\log \log x}{\varphi(m)}+B+ O\left(\frac{1}{\log x}\right).
\end{equation*}

\end{proof}
\subsubsection{Proof of Theorem \ref{thm:error-pi-mn} by Lemma \ref{lem:error-pi-mn}}
Recall the statement of Theorem \ref{thm:error-pi-mn}.
\begin{thm}\textbf{\ref{thm:error-pi-mn}}
    Fix a real number $k>-1$ and positive integers $m,n\in\bbz_{>0}$ such that $\gcd(m,n)=1$. Then
    \begin{equation*}
        \int_1^\infty \frac{\pi_k(t;m,n)-\pi(t^{k+1};m,n)}{t^{k+2}}\, dt\ =\ -\frac{\log(k+1)}{(k+1) \varphi(m)}.
    \end{equation*}

    \noindent Further,
    \begin{align*}
    -\frac{\log(k+1)}{(k+1) \varphi(m)}\ &<\ 0\qquad (k>0)\\
    -\frac{\log(k+1)}{(k+1) \varphi(m)}\ &>\ 0\qquad (-1<k<0).\\
    \end{align*}
    \end{thm}
    
    \begin{proof}
       By Lemma \ref{lem:error-pi-mn}, if $x\ge 2$, we have
        \begin{align*}
            \sum\limits_{\substack{p\le x\\ p\equiv n\pmod{m}}} \frac{1}{p}&\ =\ \frac{\log\log (x)}{\varphi(m)}+B+O\left(\frac{1}{\log x}\right) \label{recip.}\\
            &=\ \frac{1}{\varphi(m)} \log\log x+B+o(1),
        \end{align*}
         where $B$ is a constant.\\

        \noindent Let $u=t^{k+1}$. Then
        \begin{align*}
            \int_1^x \frac{\pi(t^{k+1};m,n)}{t^{k+2}}\,dt&\ =\ \frac{1}{k+1}\int_1^{x^{k+1}} \frac{\pi(u;m,n)}{u^2}\,du\\
            &=\ \frac{-1}{k+1}\frac{\pi(x^{k+1};m,n)}{x^{k+1}}+\frac{1}{k+1}\int_1^{x^{k+1}} \frac{1}{u}\, d\pi(u;m,n)\\
            &=\ \frac{-1}{k+1}\frac{\pi(x^{k+1};m,n)}{x^{k+1}}+\frac{1}{k+1}\sum\limits_{\substack{p\le x^{k+1}\\ p\equiv n\pmod{m}}} \frac{1}{p}\\
            &=\ \frac{\log\log (x^{k+1})}{(k+1)\varphi(m)}+\frac{B}{k+1}+o(1).
        \end{align*}
        Similarly,
        \begin{align*}
            \int_1^x \frac{\pi_k(t,m,n)}{t^{k+2}}\,dt&\ =\ \frac{-1}{k+1}\frac{\pi_k(x,m,n)}{x^{k+1}}+\frac{1}{k+1}\int_1^x \frac{1}{t^{k+1}}\,d\pi_k(t,m,n)\\
            &=\ \frac{-1}{k+1}\frac{\pi_k(x,m,n)}{x^{k+1}}+\frac{1}{k+1}\sum\limits_{\substack{p\le x\\ p\equiv n\pmod{m}}} \frac{1}{p^{k+1}}p^k\\
            &=\ \frac{-1}{k+1}\frac{\pi_k(x,m,n)}{x^{k+1}}+\frac{1}{k+1}\sum\limits_{\substack{p\le x\\ p\equiv n\pmod{m}}} \frac{1}{p}\\
            &=\ \frac{\log\log (x^{k+1})}{(k+1)\varphi(m)}+\frac{B}{k+1}+o(1).
        \end{align*}
        Thus, 
        \[\int_1^x \frac{\pi_k(t,m,n)-\pi(t^{k+1},m,n)}{t^{k+2}}\,dt\ =\ -\frac{\log(k+1)}{(k+1)\varphi(m)}+o(1),\] 
        which implies the theorem.
    \end{proof}

\section{Appendices}
\subsection{Appendix: Some Examples of Main Result}\label{error-cal}
In this section, we present tables containing illustrative examples that showcase the accuracy of our main results. The main goal is to view how closely $\pi_k(x^{1/(k+1)};m,n)$ approximates $\pi(x;m,n)$.\\

\noindent Utilizing High Performance Computing\footnote[1]{Appendix \ref{error-cal} is based upon work supported by the Great Lakes Slurm cluster at the University of Michigan, Ann Arbor, MI, where the author (Zhangze Li) is currently enrolled as an undergraduate student, and the National Science Foundation under Grant No. DMS-1929284 while the author (Stephanie Reyes) was in residence at the Institute for Computational and Experimental Research in Mathematics in Providence, RI, during the Summer@ICERM program.}, we examine the approximation for primes of the form $p\equiv 1 \pmod 4$, $p\equiv 3 \pmod 4$, $p\equiv 1 \pmod 5$, and $p\equiv 3 \pmod 5$. With above modular cases, we select $9$ positive integers between $10^4$ and $10^8$, and calculate
\[\text{Error}\ (x,k;m,n)\ := \ \frac{\pi(x;m,n)-\pi_k(x^{1/(k+1)};m,n)}{\pi(x;m,n)}\]
for $k=1,\ 1/2,\ -1/10,\ -1/12$.

\subsubsection{Error for $k=1$}
We observe that of the 36 Errors measured in Tables 1-4, $72.22\%$ are positive.

\begin{table}[H]
\begin{center}
\begin{tabular}{||c| c| c| c||} 
 \hline
 $x$ & $\pi(x;4,1)$ & $\pi_1(x^{1/2};4,1)$ &  Error $\%$\\ [0.5ex]
 \hline\hline
 $1\times 10^4$ & 609 & 515 & $15.43514\%$ \\ 
 \hline
 $5\times 10^4$ & 2549 & 2025 & $20.55708 \%$ \\ 
 \hline
 $1\times 10^5$ & 4783& 4418 & $7.63119\%$ \\
 \hline
 $5\times 10^5$ & 20731& 19668 & $5.12759 \%$ \\
 \hline
 $1\times 10^6$ & 39175 & 36628 & $6.50160 \%$ \\
 \hline
  $5\times 10^6$ & 174193 & 165373 & $5.06335\%$ \\
 \hline
 $1\times 10^7$ &332180& 323048 & $2.74911\%$ \\ 
 \hline
 $5\times 10^7$ &1500452& 1475230 & $1.68096\%$ \\ 
 \hline
 $1\times 10^8$ & 2880504 & 2863281 & $0.59792\%$ \\ [1ex]
 \hline
\end{tabular}
\caption{$\pi(x;4,1)$ and $\pi_{1}(x^{1/2};4,1)$.}
\end{center}
\end{table}

\begin{table}[H]
\begin{center}
\begin{tabular}{||c| c| c| c||} 
 \hline
 $x$ & $\pi(x;4,3)$ & $\pi_1(x^{1/2};4,3)$ &  Error $\%$\\ [0.5ex]
 \hline\hline
 $1\times 10^4$ & 619 & 543 & { }$12.27787\%$ \\ 
 \hline
 $5\times 10^4$ & 2583 & 2411 &{ } $6.65892\%$ \\ 
 \hline
 $1\times 10^5$ & 4808& 4786 &{ } $0.45757\%$ \\
 \hline
 $5\times 10^5$ & 20806& 20643 & { } $0.78343\%$ \\
 \hline
 $1\times 10^6$ & 39322 & 39497 & $-0.44504\%$ \\
 \hline
  $5\times 10^6$ & 174319 & 170667 &{ } $2.09501\%$ \\
 \hline
 $1\times 10^7$ &332398& 319819 &{ } $3.78432\%$ \\ 
 \hline
 $5\times 10^7$ &1500681& 1516507 & $-1.05459\%$ \\ 
 \hline
 $1\times 10^8$ & 2880950 & 2873113 &{ } $0.27203\%$ \\ [1ex]
 \hline
\end{tabular}
\caption{$\pi(x;4,3)$ and $\pi_{1}(x^{1/2};4,3)$.}
\end{center}
\end{table}

\begin{table}[H]
\begin{center}
\begin{tabular}{||c| c| c| c||} 
 \hline
 $x$ & $\pi(x;5,1)$ & $\pi_1(x^{1/2};5,1)$ &  Error $\%$\\ [0.5ex]
 \hline\hline
 $1\times 10^4$ & 306 & 215 &{ }$29.73856\%$ \\ 
 \hline
 $5\times 10^4$ & 1274 & 1181 &{ } $7.29984\%$ \\ 
 \hline
 $1\times 10^5$ & 2387& 2536 & $-6.24214\%$ \\
 \hline
 $5\times 10^5$ & 10386& 10631 & $-2.35894\%$ \\
 \hline
 $1\times 10^6$ & 19617 & 18470 &{ } $5.84697 \%$ \\
 \hline
  $5\times 10^6$ & 87062 & 81830 &{ } $6.00951\%$ \\
 \hline
 $1\times 10^7$ &166104& 156148 &{ } $5.99384\%$ \\ 
 \hline
 $5\times 10^7$ &750340& 763457 & $-1.74814\%$ \\ 
 \hline
 $1\times 10^8$ & 1440298 & 1448386 & $-0.56155\%$ \\ [1ex]
 \hline
\end{tabular}
\caption{$\pi(x;5,1)$ and $\pi_{1}(x^{1/2};5,1)$.}
\end{center}
\end{table}

\begin{table}[H]
\begin{center}
\begin{tabular}{||c| c| c| c||} 
 \hline
 $x$ & $\pi(x;5,3)$ & $\pi_1(x^{1/2};5,3)$ &  Error $\%$\\ [0.5ex]
 \hline\hline
 $1\times 10^4$ & 310 & 291 &{ } $6.12903\%$ \\ 
 \hline
 $5\times 10^4$ & 1290 & 1036 &{ }$19.68992\%$ \\ 
 \hline
 $1\times 10^5$ & 2402& 2644 &$-10.07494\%$ { } \\
 \hline
 $5\times 10^5$ & 10382& 10539 & $-1.51223\%$ \\
 \hline
 $1\times 10^6$ & 19665 & 18146 &{ } $7.72438\%$ \\
 \hline
  $5\times 10^6$ & 87216 & 90461 & $-3.72065\%$ \\
 \hline
 $1\times 10^7$ &166230& 156456 &{ } $5.87981\%$ \\ 
 \hline
 $5\times 10^7$ &750395& 753820 & $-0.45643\%$ \\ 
 \hline
 $1\times 10^8$ & 1440474 & 1436510 &{ } $0.27519\%$ \\ [1ex]
 \hline
\end{tabular}
\caption{$\pi(x;5,3)$ and $\pi_{1}(x^{1/2};5,3)$.}
\end{center}
\end{table}

\newpage
\subsubsection{Error for $k=1/2$}
We observe that of the 36 Errors measured in Tables 5-8, $72.22\%$ are positive.

\begin{table}[H]
\begin{center}
\begin{tabular}{||c| c| c| c||} 
 \hline
 $x$ & $\pi(x;4,1)$ & $\pi_{1/2}(x^{2/3};4,1)$ &  Error $\%$\\ [0.5ex]
 \hline\hline
 $1\times 10^4$ & 609 & 617.62512 &$-1.41628\%$ { } \\ 
 \hline
 $5\times 10^4$ & 2549& 2477.64505 & $2.79933\%$ \\
 \hline
 $1\times 10^5$ & 4783& 4659.83812 & $2.57499\%$ \\
 \hline
 $5\times 10^5$ & 20731& 20125.89212 & $2.91886\%$ \\
 \hline
 $1\times 10^6$ & 39175 & 38904.00140 & $0.69176\%$ \\
 \hline
 $5\times 10^6$ & 174193 & 173246.23939 & $0.54351\%$ \\
 \hline
 $1\times 10^7$ &332180& 329252.45078 & $0.88131\%$ \\ 
 \hline
 $5\times 10^7$ &1500452& 1492885.30185 & $0.50429\%$ \\ 
 \hline
 $1\times 10^8$ & 2880504 & 2873027.62482 & $0.25955\%$ \\ [1ex]
 \hline
\end{tabular}
\caption{$\pi(x;4,1)$ and $\pi_{1/2}(x^{2/3};4,1)$.}
\end{center}
\end{table}

\begin{table}[H]
\begin{center}
\begin{tabular}{||c| c| c| c||} 
 \hline
 $x$ & $\pi(x;4,3)$ & $\pi_{1/2}(x^{2/3};4,3)$ &  Error $\%$\\ [0.5ex]
 \hline\hline
 $1\times 10^4$ & 619 & 591.60159 & $4.42624\%$ \\ 
 \hline
 $5\times 10^4$ &2583 & 2502.18366 & $3.12878\%$ \\
 \hline
 $1\times 10^5$ &4808  & 4833.35209 & $-0.52729\%$ { } \\
 \hline
 $5\times 10^5$ & 20806 & 20951.89316 & $-0.70121\%$ { } \\
 \hline
 $1\times 10^6$ & 39322 & 39178.87051 & $0.36399\%$ \\
 \hline
 $5\times 10^6$ & 174319 & 173924.73741 & $0.22617\%$ \\
 \hline
 $1\times 10^7$ & 332398 & 331806.98445 & $0.17780\%$ \\ 
 \hline
 $5\times 10^7$ &1500681 & 1502046.79913 & $-0.09101\%$ { } \\ 
 \hline
 $1\times 10^8$ & 2880950 & 2879155.53993 & $0.06229\%$ \\ [1ex]
 \hline
\end{tabular}
\caption{$\pi(x;4,3)$ and $\pi_{1/2}(x^{2/3};4,3)$.}
\end{center}
\end{table}

\begin{table}[H]
\begin{center}
\begin{tabular}{||c| c| c| c||} 
 \hline
 $x$ & $\pi(x;5,1)$ & $\pi_{1/2}(x^{2/3};5,1)$ &  Error $\%$\\ [0.5ex]
 \hline\hline
 $1\times 10^4$ & 306 & 290.30286 & $5.12978\%$ \\ 
 \hline
 $5\times 10^4$ &1274 & 1243.85408 & $2.36624\%$ \\
 \hline
 $1\times 10^5$ & 2387 & 2288.69057 & $4.11853\%$ \\
 \hline
 $5\times 10^5$ & 10386 & 10309.63049 & $0.73531\%$ \\
 \hline
 $1\times 10^6$ & 19617 & 19616.30635 & $0.00354\%$ \\
 \hline
 $5\times 10^6$ & 87062 & 87036.61969 & $0.02915\%$ \\
 \hline
 $1\times 10^7$ & 166104 & 164310.69864 & $1.07963\%$ \\ 
 \hline
 $5\times 10^7$ & 750340 & 752249.09877 & $-0.25443\%$ { } \\ 
 \hline
 $1\times 10^8$ & 1440298 & 1430300.15946 & $0.69415\%$ \\ [1ex]
 \hline
\end{tabular}
\caption{$\pi(x;5,1)$ and $\pi_{1/2}(x^{2/3};5,1)$.}
\end{center}
\end{table}

\begin{table}[H]
\begin{center}
\begin{tabular}{||c| c| c| c||} 
 \hline
 $x$ & $\pi(x;5,3)$ & $\pi_{1/2}(x^{2/3};5,3)$ &  Error $\%$\\ [0.5ex]
 \hline\hline
 $1\times 10^4$ & 310 & 321.30898 & $-3.64806\%$ \\ 
 \hline
 $5\times 10^4$ &1290 & 1244.44539 &{ } $3.53136\%$ \\
 \hline
 $1\times 10^5$ &2402  &  2501.69252& $-4.15040\%$ \\
 \hline
 $5\times 10^5$ &10382  & 10504.71338 & $-1.18198\%$ \\
 \hline
 $1\times 10^6$ & 19665 &19604.34768  & { } $0.30843 \%$ \\
 \hline
 $5\times 10^6$ & 87216 & 86272.37280 &{ } $1.08194\%$ \\
 \hline
 $1\times 10^7$ & 166230 &167998.78804  & $-1.06406\%$ \\ 
 \hline
 $5\times 10^7$ & 750395& 748916.40906 &{ } $0.19704\%$ \\ 
 \hline
 $1\times 10^8$ & 1440474 & 1444351.68992 & $-0.26920\%$ \\ [1ex]
 \hline
\end{tabular}
\caption{$\pi(x;5,3)$ and $\pi_{1/2}(x^{2/3};5,3)$.}
\end{center}
\end{table}

\subsubsection{Error for $k=-1/10$}
We observe that of the 36 Errors measured in Tables 9-12, $72.22\%$ are negative.

\begin{table}[H]
\begin{center}
\begin{tabular}{||c| c| c| c||} 
 \hline
 $x$ & $\pi(x;4,1)$ & $\pi_{-1/10}(x^{10/9};4,1)$ &  Error $\%$\\ [0.5ex]
 \hline\hline
 $1\times 10^4$ & 609 & 613.50169 & $-0.73919 \%$ \\ 
 \hline
 $5\times 10^4$ & 2549 & 2562.89963 & $-0.54530 \%$ \\ 
 \hline
 $1\times 10^5$ & 4783& 4788.03485 & $-0.10527 \%$ \\
 \hline
 $5\times 10^5$ & 20731& 20771.18437 & $-0.19384 \%$ \\
 \hline
 $1\times 10^6$ & 39175 & 39266.51644 & $-0.23361 \%$ \\
 \hline
  $5\times 10^6$ & 174193 & 174232.64634 & $-0.02276 \%$ \\
 \hline
 $1\times 10^7$ &332180& 332314.25320 & $-0.04042 \%$ \\ 
 \hline
 $5\times 10^7$ &1500452& 1500545.39963 & $-0.00622 \%$ \\ 
 \hline
 $1\times 10^8$ & 2880504 &  2880813.47274 & $-0.01074 \%$ \\ [1ex]
 \hline
\end{tabular}
\caption{$\pi(x;4,1)$ and $\pi_{-1/10}(x^{10/9};4,1)$.}
\end{center}
\end{table}

\begin{table}[H]
\begin{center}
\begin{tabular}{||c| c| c| c||} 
 \hline
 $x$ & $\pi(x;4,3)$ & $\pi_{-1/10}(x^{10/9};4,3)$ &  Error $\%$\\ [0.5ex]
 \hline\hline
 $1\times 10^4$ & 619 & 618.68563 &{ } $0.05079 \%$ \\ 
 \hline
 $5\times 10^4$ & 2583 & 2579.32406 &{ } $0.14231 \%$ \\ 
 \hline
 $1\times 10^5$ & 4808& 4821.56790 & $-0.28219 \%$ \\
 \hline
 $5\times 10^5$ & 20806& 20796.90807 &{ } $0.04370 \%$ \\
 \hline
 $1\times 10^6$ & 39322 & 39305.82276 &{ } $0.04114 \%$ \\
 \hline
  $5\times 10^6$ & 174319 & 174323.13838 & $-0.00237 \%$ \\
 \hline
 $1\times 10^7$ &332398& 332477.68215 & $-0.02397 \%$ \\ 
 \hline
 $5\times 10^7$ &1500681& 1500779.09018 & $-0.00654 \%$ \\ 
 \hline
 $1\times 10^8$ & 2880950 & 2881063.76609 & $-0.00395 \%$ \\ [1ex]
 \hline
\end{tabular}
\caption{$\pi(x;4,3)$ and $\pi_{-1/10}(x^{10/9};4,3)$.}
\end{center}
\end{table}

\begin{table}[H]
\begin{center}
\begin{tabular}{||c| c| c| c||} 
 \hline
 $x$ & $\pi(x;5,1)$ & $\pi_{-1/10}(x^{10/9};5,1)$ &  Error $\%$\\ [0.5ex]
 \hline\hline
 $1\times 10^4$ & 306 & 306.84917 & $-0.27750 \%$ \\ 
 \hline
 $5\times 10^4$ & 1274 & 1284.72538 & $-0.84187 \%$ \\ 
 \hline
 $1\times 10^5$ & 2387& 2397.35493 & $-0.43380 \%$ \\
 \hline
 $5\times 10^5$ & 10386 & 10381.05165 & $-0.04764 \%$ \\
 \hline
 $1\times 10^6$ & 19617 & 19624.30360 & $-0.03723 \%$ \\
 \hline
  $5\times 10^6$ & 87062 & 87119.41013 & $ -0.06594 \%$ \\
 \hline
 $1\times 10^7$ &166104& 166161.85950 & $-0.03483 \%$ \\ 
 \hline
 $5\times 10^7$ &750340& 750274.35740 &{ } $0.00875 \%$ \\ 
 \hline
 $1\times 10^8$ & 1440298 & 1440380.00374 & $-0.00569 \%$ \\ [1ex]
 \hline
\end{tabular}
\caption{$\pi(x;5,1)$ and $\pi_{-1/10}(x^{10/9};5,1)$.}
\end{center}
\end{table}

\begin{table}[H]
\begin{center}
\begin{tabular}{||c| c| c| c||} 
 \hline
 $x$ & $\pi(x;5,3)$ & $\pi_{-1/10}(x^{10/9};5,3)$ &  Error $\%$\\ [0.5ex]
 \hline\hline
 $1\times 10^4$ & 310 & 309.65259 & $0.11207 \%$ \\ 
 \hline
 $5\times 10^4$ & 1290 & 1288.73787 & $ 0.09784 \%$ \\ 
 \hline
 $1\times 10^5$ & 2402& 2406.10885 & $ -0.17106 \%$ { } \\
 \hline
 $5\times 10^5$ & 10382& 10404.21103 & $-0.21394 \%$ { } \\
 \hline
 $1\times 10^6$ & 19665 & 19658.64066 & $0.03234 \%$ \\
 \hline
  $5\times 10^6$ & 87216 & 87152.03182 & $0.07334 \%$ \\
 \hline
 $1\times 10^7$ &166230& 166229.30179 & $0.00042 \%$ \\ 
 \hline
 $5\times 10^7$ &750395& 750428.14712 & $-0.00442 \%$ { } \\ 
 \hline
 $1\times 10^8$ & 1440474 & 1440534.90175 & $-0.00423 \%${ } { } \\ [1ex]
 \hline
\end{tabular}
\caption{$\pi(x;5,3)$ and $\pi_{-1/10}(x^{10/9};5,3)$.}
\end{center}
\end{table}

\subsubsection{Error for $k=-1/12$}
We observe that of the 36 Errors measured in Tables 13-16, $77.78\%$ are negative.

\begin{table}[H]
\begin{center}
\begin{tabular}{||c| c| c| c||} 
 \hline
 $x$ & $\pi(x;4,1)$ & $\pi_{-1/12}(x^{12/11};4,1)$ &  Error $\%$\\ [0.5ex]
 \hline\hline
 $1\times 10^4$ & 609 & 611.17719 & $-0.35750 \%$ \\ 
 \hline
 $5\times 10^4$ & 2549& 2558.04851 & $-0.35498 \%$ \\
 \hline
 $1\times 10^5$ & 4783& 4787.40169 & $-0.09203 \%$ \\
 \hline
 $5\times 10^5$ & 20731& 20762.97056 & $-0.15422 \%$ \\
 \hline
 $1\times 10^6$ & 39175 & 39253.59228  & $ -0.20062 \%$ \\
 \hline
 $5\times 10^6$ & 174193 & 174211.69780  & $-0.01073 \%$ \\
 \hline
 $1\times 10^7$ &332180&  332319.22077  & $ -0.04191 \%$ \\ 
 \hline
 $5\times 10^7$ &1500452&  1500497.92446 & $-0.00306 \%$ \\ 
 \hline
 $1\times 10^8$ & 2880504 & 2880766.12283  & $-0.00910 \%$ \\ [1ex]
 \hline
\end{tabular}
\caption{$\pi(x;4,1)$ and $\pi_{-1/12}(x^{12/11};4,1)$.}
\end{center}
\end{table}

\begin{table}[H]
\begin{center}
\begin{tabular}{||c| c| c| c||} 
 \hline
 $x$ & $\pi(x;4,3)$ & $\pi_{-1/12}(x^{12/11};4,3)$ &  Error $\%$\\ [0.5ex]
 \hline\hline
 $1\times 10^4$ & 619 & 622.36367 & $-0.54340 \%$ \\ 
 \hline
 $5\times 10^4$ &2583 & 2584.93696 & $-0.07499 \%$ \\
 \hline
 $1\times 10^5$ &4808  & 4816.96821 & $-0.18653 \%$ \\
 \hline
 $5\times 10^5$ & 20806 & 20788.86077 &{ } $ 0.08238 \%$ \\
 \hline
 $1\times 10^6$ & 39322 & 39306.76238 &{ } $0.03875 \%$ \\
 \hline
 $5\times 10^6$ & 174319 & 174325.21822 & $-0.00357 \%$ \\
 \hline
 $1\times 10^7$ & 332398 & 332443.67182 & $-0.01374 \%$ \\ 
 \hline
 $5\times 10^7$ &1500681 & 1500833.67573 & $-0.01017 \%$ \\ 
 \hline
 $1\times 10^8$ & 2880950 & 2881059.64847 & $-0.00381 \%$ \\ [1ex]
 \hline
\end{tabular}
\caption{$\pi(x;4,3)$ and $\pi_{-1/12}(x^{12/11};4,3)$.}
\end{center}
\end{table}

\begin{table}[H]
\begin{center}
\begin{tabular}{||c| c| c| c||} 
 \hline
 $x$ & $\pi(x;5,1)$ & $\pi_{-1/12}(x^{12/11};5,1)$ &  Error $\%$\\ [0.5ex]
 \hline\hline
 $1\times 10^4$ & 306 & 308.32261 & $-0.75902 \%$ \\ 
 \hline
 $5\times 10^4$ &1274 & 1286.24176 & $ -0.96089 \%$ \\
 \hline
 $1\times 10^5$ & 2387 & 2392.89814 & $ -0.24709  \%$ \\
 \hline
 $5\times 10^5$ &  10386 & 10367.25069  &{ } $0.18052 \%$ \\
 \hline
 $1\times 10^6$ & 19617 &  19631.27471 & $-0.07277 \%$ \\
 \hline
 $5\times 10^6$ & 87062 & 87117.75683  & $-0.06404 \%$ \\
 \hline
 $1\times 10^7$ & 166104 &  166145.81965 & $ -0.02518 \%$ \\ 
 \hline
 $5\times 10^7$ & 750340 & 750274.07185  &{ } $0.00879 \%$ \\ 
 \hline
 $1\times 10^8$ & 1440298 &  1440333.42281 & $-0.00246 \%$ \\ [1ex]
 \hline
\end{tabular}
\caption{$\pi(x;5,1)$ and $\pi_{-1/12}(x^{12/11};5,1)$.}
\end{center}
\end{table}

\begin{table}[H]
\begin{center}
\begin{tabular}{||c| c| c| c||} 
 \hline
 $x$ & $\pi(x;5,3)$ & $\pi_{-1/12}(x^{12/11};5,3)$ &  Error $\%$\\ [0.5ex]
 \hline\hline
 $1\times 10^4$ & 310 & 310.43590 & $-0.14061 \%$ \\ 
 \hline
 $5\times 10^4$ &1290 & 1290.55374 & $-0.04293 \%$ \\
 \hline
 $1\times 10^5$ &2402  & 2406.74513 & $-0.19755 \%$ \\
 \hline
 $5\times 10^5$ &10382  & 10393.81954  & $-0.11385 \%$ \\
 \hline
 $1\times 10^6$ & 19665 &  19662.29214 &{ } $0.01377  \%$ \\
 \hline
 $5\times 10^6$ & 87216 & 87159.52414  &{ } $ 0.06475 \%$ \\
 \hline
 $1\times 10^7$ & 166230 & 166209.29782  &{ } $0.01245 \%$ \\ 
 \hline
 $5\times 10^7$ & 750395&  750355.12871 &{ } $0.00531 \%$ \\ 
 \hline
 $1\times 10^8$ & 1440474 & 1440579.59405  & $-0.00733 \%$ \\ [1ex]
 \hline
\end{tabular}
\caption{$\pi(x;5,3)$ and $\pi_{-1/12}(x^{12/11};5,3)$.}
\end{center}
\end{table}

\subsubsection{Conclusion} 
 Note that when $k \ > \ 0$, Error exhibits a positive trend, indicating that $\pi_(x^{k+1};m,n)$ tends to be larger $\pi_k(x;m,n)$. The opposite trend holds when $k \ < \ 0$. The explanation for this phenomenon is provided by Theorem \ref{thm:error-pi-mn} and Remark \ref{remark_1}.

\subsection{Appendix: Riemann–Stieltjes Integration}
\label{append-int}
\subsubsection{Riemann-Stieltjes Integral}
Let $f$ be a real-valued function of a real variable on $[a,b]$ where $a,b\in \mathbb{R}$, and let $g$ also be a real-to-real function. Let $P=\{a=x_0 < x_1 < x_2 < \dots < x_n =b\}$ be a partition of $[a,b]$. Consider
\begin{equation*}
    S(f,g,P)\ := \ \sum_{i=1}^{n} f(x_{i}^*) [g(x_{i})-g(x_{i-1})]
\end{equation*}
where $x_{i}^*\in [x_{i-1},x_{i}].$\\

\noindent If $S(f,g,P)$ converges to a constant $L$ when $\max\limits_{1\le i \le n} \{x_{i}-x_{i-1}\}$ approaches $0$, then we define Riemann-Stieltjes Integral as following:
\begin{equation*}
    \int_{a}^{b} f\, dg \ := \ L.
\end{equation*}
The Riemann-Steltjes integral exists if  $g$ is of bounded variation and right-semicontinuous. 
\subsubsection{Applying Riemann-Stieltjes Integration}
Consider
\begin{equation*}
    \pi(x;m,n)-\pi(x-1;m,n) \ =\  \begin{cases}
            1,& \text{if $x$ is prime and $x\ \equiv\  n \Mod m$}, \\
            \\
            0, &\substack{\ds\text{if $x$ is composite or}\\ 
            \ds\text{$x$ is prime where $x\ \not\equiv\  n \Mod m$.}}
            \end{cases}
\end{equation*}
Let $f$ be a real-valued function of a real variable, then 
\begin{equation*}
   \sum_{\substack{p \le x\\ p\  \equiv\  n\Mod{m}}}f(p)\qquad =\qquad \sum_{i=2}^{x} f(i) [\pi(x;m,n)-\pi(x-1;m,n)].
\end{equation*}
By applying the Riemann-Stieltjes Integral, we have
\begin{align}
\nonumber
    \sum_{\substack{p \le x\\ p\  \equiv\  n\Mod{m}}}f(p)\qquad &=\qquad \sum_{i=2}^{x} f(i) [\pi(x;m,n)-\pi(x-1;m,n)]\\ 
    &=\qquad \int_{2}^{x} f(x)\, d\pi(x;m,n).\label{riea-stieljes}
\end{align}
This technique is crucial to evaluate the Riemann-Stieltjes Integral in proofs of Theorem \ref{1:thm} and Theorem \ref{thm:error-pi-mn}.

\subsection{Appendix: Limit Computation}
\label{append-lim}
From Gerard-Washington \cite[Lemma 1]{GW}, we have the following result.\\

\noindent Let $k>-1$ and let $f$ be a differentiable function on $[2,\infty)$ such that 
\begin{enumerate}[label=\roman*.]
    \item $\ds x^{k+1}\exp(-f(x))\to\infty$\quad and 
    \item $\ds x f'(x)\to 0$
\end{enumerate} 
as $x\to\infty$. Then
$\ds\int_{2}^{x}t^k \exp(-f(t))\, dt\ =\ O\left(x^{k+1} \exp(-f(x)) \right)$.\\

\noindent We apply this result with $f(x)=\frac{1}{2}\alpha \sqrt{\log x}$, where $\alpha$ is a positive constant.\\

\noindent We verified the hypotheses in our case by computing the relevant limits for applying the lemma.
\begin{enumerate}[label=\roman*.]
    \item We will find
    \begin{equation*}
 L\ :=\ \lim_{x\to\infty}\left[x^{k+1}\exp\left(-\frac{1}{2}\alpha\sqrt{\log x}\right) \right].
    \end{equation*}

\noindent Let
\[x \ = \ \exp\left(\frac{4y^2}{\alpha^2}\right).\]
When $ y\to\infty$, $ x \to \infty$. Thus,
\begin{align*}
    L \ &= \ \lim_{y\to\infty}\left[\left[\exp\left(\frac{4y^2}{\alpha^2}\right)\right]^{k+1}\exp\left(-\frac{1}{2}\alpha \cdot \frac{2y}{\alpha}\right) \right] \\
    &= \ \lim_{y\to\infty}\left[\exp\left((k+1)\cdot \frac{4y^2}{\alpha^2}\right)\exp\left(-y\right) \right]\\
    &= \ \lim_{y\to\infty}\left[\exp\left(\frac{4k+4}{\alpha^2}\cdot y^2\right)\exp\left(-y\right) \right].
\end{align*}
Since $ k\ >\ -1$, we have $k\ +\ 1\ >\ 0  \Longrightarrow \ 4k\ +\ 4\ >\ 0.$ Hence,
\[\frac{4k+4}{\alpha^2} \ >\ 0,\]
which indicates that
\[y \ \lll \frac{4k+4}{\alpha^2} \cdot y^2 \]
when $y$ approaches $\infty$.\\

\noindent Thus, $\exp\left(\frac{4k+4}{\alpha^2} \cdot y^2\right)$ grows much faster than $\exp\left(-y\right)$ decays, so we conclude that 
\[L \ = \ \lim_{y\to\infty}\left[\exp\left(\frac{4k+4}{\alpha^2}\cdot y^2\right)\exp\left(-y\right) \right] \ = \ \infty.\]
Therefore,
\begin{equation*}
    \lim_{x\to\infty}\left[x^{k+1}\exp\left(-\frac{1}{2}\alpha\sqrt{\log x}\right) \right]\ =\ \infty.
\end{equation*}
\\

    \item Recall that $\ds f(x)\ =\ \frac{1}{2}\alpha\sqrt{\log x}$. Thus
    \begin{align*}
    f'(x)\ &=\ \frac{\alpha}{4x\sqrt{\log x}}\\
    \imply x f'(x)\ &=\ \frac{\alpha}{4\sqrt{\log x}}
    \end{align*}
Since\quad $\ds\lim_{x\to \infty}\sqrt{\log x}\ =\ \infty$,\quad it follows that
\[\lim_{x\to\infty}[x f'(x)]\ =\ 0.\]
\end{enumerate}

\end{document}